\documentclass[12pt]{amsart}
\usepackage{amsmath}
\usepackage{mathtools}

\usepackage{amsfonts}
\usepackage{amsthm}
\usepackage{autobreak}
\usepackage[english]{babel}
\usepackage{bold-extra}
\usepackage{hyperref}
\usepackage{comment}
\usepackage{mdwlist}
\theoremstyle{definition}
\newtheorem{definition}{Definition}
\newtheorem*{problem}{Problem}
\newtheorem{remark}[definition]{Remark}
\theoremstyle{plain}
\newtheorem{prop}[definition]{Proposition}
\newtheorem{notations}[definition]{Notations}
\newtheorem{cor}[definition]{Corollary}

\newtheorem{lemma}[definition]{Lemma}
\newtheorem{theorem}[definition]{Theorem}
\newtheorem*{maintheorem}{Main theorem}

\newtheorem*{questions}{Open questions}

\newtheorem{fact}[definition]{Fact}
\numberwithin{definition}{section}
\allowdisplaybreaks

\def\D{{\mathbb{D}}}

\def\R{{\mathbb{R}}}
\def\C{{\mathbb{C}}}
\def\Q{{\mathbb{Q}}}
\def\N{{\mathbb{N}}}

\def\T{{\mathbb{T}}}
\def\UU{{\mathcal{U}}}

\def\LL{{\mathcal{L}}}

\def\EE{{\mathcal{E}}}

\def\PP{{\mathcal{P}}}
\def\CC{{\mathcal{C}}}

\numberwithin{equation}{section} 
\numberwithin{figure}{section} 
\numberwithin{table}{section} 

\newcommand{\vp}{\varphi}

\usepackage{multicol}
\usepackage[T1]{fontenc}
\usepackage[utf8]{inputenc}
\usepackage[a4paper]{geometry}
\usepackage{amssymb,amsmath}
\usepackage{version}
\usepackage{mathtools} 
\usepackage{enumerate}
\usepackage{dsfont}
\usepackage{mathdots}
\makeatletter

\usepackage{latexsym}
\usepackage{color}
\begin{document}
	
	\title{New results on universal Taylor series via weighted polynomial approximation}

	\author{St\'{e}phane Charpentier, Konstantinos Maronikolakis}
	
	\address{St\'ephane Charpentier, Aix-Marseille Univ, CNRS, I2M, Marseille, France}
	\email{stephane.charpentier.1@univ-amu.fr}
	
	\address{Konstantinos Maronikolakis, Department of Mathematics, Bilkent University, 06800 Ankara, Turkey}
	\email{conmaron@gmail.com}
	
	\thanks{The research was conducted during the stay of the second author in the Institut de Mathématiques de Marseille, and was funded by the Embassy of France in Ireland through the 2024 France Excellence Research Residency grant. Konstantinos Maronikolakis also acknowledges financial support from the Irish Research Council through the grant [GOIPG/2020/1562]. The first author is partially supported by the grant ANR-24-CE40-0892-01 of the French National Research Agency
		ANR (project ComOp).}
	
	
	
	\begin{abstract}We use weighted polynomial approximation to prove the existence of a compact set $K$ with non-empty interior and a function $f(z)=\sum_k a_kz^k$ holomorphic in $\D$, such that the set
		\[
		\left\{\sum_{k=0}^n a_kz^k:\,n\in \N\right\}\cup \left\{-\sum_{k=0}^na_kz^k:\,n\in \N\right\}
		\]
		is dense in the space $A(K)$ of all continuous functions on $K$ that are holomorphic in the interior of $K$, endowed with the $\sup$ norm, while the set $\{\sum_{k=0}^n a_kz^k:\,n\in \N\}$ is not. This improves a result of Mouze. The main ideas of the proof also allows us to construct a holomorphic function $f(z)=\sum_k a_kz^k$, such that $\{\sum_{k=0}^na_kz^k:\,n\in \N\}$ is dense in $A(K)$, while the modulus of its non-zero Taylor coefficients go to $\infty$.
		
		In passing, we complement a result by Pritsker and Varga on weighted polynomial approximation by proving that, for any compact set $K$ with connected complement, there exists a constant $\alpha_K >0$ such that there exists a bounded domain $G$ containing $K$ such that the weighted polynomials of the form $z^{\alpha n}P_n$, with $\text{deg}(P_n)\leq n$, are dense in $H(G)$ for the topology of locally uniform convergence if and only if $\alpha<\alpha_K$. Explicit computations of $\alpha_K$ are given for some simple compact sets $K$.
	\end{abstract}
	
	\maketitle
	
	\section{Introduction}
	
	The motivation for this paper comes from three classical results in Linear Dynamics. Given two Fréchet spaces $X$ and $Y$, we denote by $\LL(X,Y)$ the space of all bounded linear operators (we will write $\LL(X)$ instead of $\LL(X,X)$). If $T\in \LL(X)$, we say that $x$ in $X$ is \textit{hypercyclic} for $T$ if its orbit $\{T^nx:\,n\in \N\}$ under $T$ is dense in $X$. In turn, $T$ is called \textit{hypercyclic} if such a vector $x$ exists. This notion has a natural extension to more general sequences of operators $T_n:X\to Y$; a vector $x$ in $X$ is called \textit{universal} for $(T_n)_n$ if the set $\{T_n x:\, n\in \N\}$ is dense in $Y$. The sequence $(T_n)_n$ is called \textit{universal} if there is such a vector $x$. Linear Dynamics is the field that studies the dynamical properties of linear operators including the notion of hypercyclicity defined above. Three classical results in the area are the following:
	\begin{theorem}[Costakis \cite{Costakis2000_2} and Peris \cite{Peris2001}]\label{thmCP}For any Fréchet space $X$, any $T\in \LL(X)$ and any pair $(x_1,x_2)\in X\times X$, if the set $\{T^n x_i:\, n\in \N, i\in \{1,2\}\}$	is dense in $X$, then $x_1$ or $x_2$ is hypercyclic for $T$.
	\end{theorem}
	An improvement of the previous is the following beautiful result:
	\begin{theorem}[Bourdon and Feldman \cite{BourdonFeldman2003}]\label{thmBF}For any Fréchet space $X$, any $T\in \LL(X)$ and any $x\in X$, if the set $\{T^n x:\,n\in \N\}$ is somehwere dense in $X$, then $x$ is hypercyclic for $T$.
	\end{theorem}
	In the same vein, but not a consequence of the previous, the following is also classical:
	\begin{theorem}[Leon-Saavedra and M\"uller \cite{LeonMuller2004}]\label{thmLM}For any Fréchet space $X$ and any $T\in \LL(X)$, a vector $x$ in $X$ is hypercyclic for $T$ if and only if the set
		\[
		\{T^n (\lambda x):\, n\in \N,\lambda \in \T\}
		\]
		is dense in $X$, where $\T$ denotes the unit circle in the complex plane.
	\end{theorem}
	We mention that there are simple examples of operators $T$ acting on $\ell^2(\N)$ that are \textit{supercyclic} - namely, there exists $x\in\ell^2(\N)$ such that $\{\lambda T^n x:\, n\in \N,\,\lambda \in \C\}$ is dense in $\ell^2(\N)$, while $T$ is not hypercyclic. In \cite{CharpentierErnstMenet2016}, a characterisation of the sets $\Gamma \subset \C$ for which Theorem \ref{thmLM} holds if we replace $\T$ by $\Gamma$ was obtained. We refer the reader to the textbooks \cite{BayartMatheron2009,GrosseErdmannPerisManguillot2011} for a detailed account on Linear Dynamics, before 2011.
	
	At first sight, it may seem natural to wonder to which extent these three results have a generalisation to sequences of operators that are not iterates of a single one. However, it is not difficult to artificially construct operators $T_n: X \to Y$, $n\in \N$, such that $(T_n)_n$ is a counter-example to any reasonable extension of these statements, which makes the search for generalization in this direction vacuous. Rather, it is more relevant to examine whether the properties exhibited in these theorems are satisfied by specific examples of universal sequences of operators, or not.
	
	It turns out that few examples of sequences of operators  not pertaining to Linear Dynamics have been thoroughly studied so far. The most studied universal objects are those associated with partial sum operators, that can be acting from and to various Fréchet spaces. These objects fall within the well-developed theory of \textit{universal series}, see \cite{GrosseErdmann1999,BayartGrosseErdmanNestoridisPapadimitropoulos2008}. The first example of universal series goes back to Fekete \cite{Pal1914}, who proved, given $a>0$, the existence of a formal real power series $f=\sum_ {k\geq 1}a_k x^k$ such that every continuous function on $[-a,a]$, that vanishes at $0$, can be approximated by some sequence of partial sums $\sum_{k=1}^n a_k x^k$ of $f$, uniformly on $[-a,a]$; see for example \cite{BayartGrosseErdmanNestoridisPapadimitropoulos2008} for further details. With the above notation, this result tells us that the sequence $(S_n)_n$ is universal, where $S_n$ maps an element $\sum_{k\geq 1}a_k x^k$ of the Fréchet space of formal real power series $\R[[X]]$, to the polynomial $\sum_{k=1}^n a_k x^k$, seen as an element of the Banach space $\CC_0([-a,a])$ of all continuous functions on $[-a,a]$, that vanish at $0$, endowed with the uniform topology. In \cite{Mouze2020}, the author built two formal power series $f_1$ and $f_2$ that are not universal for $(S_n)_n$, while the set
	\[
	\{S_n(f_1):\,n\in \N\}\cup \{S_n(f_2):\,n\in \N\}
	\]
	is dense in $\CC_0([-a,a])$. Thus, Mouze showed that Theorems \ref{thmCP} and \ref{thmBF} do not hold true if $(T^n)_n$ is replaced by $(S_n)_n$ acting from $\R[[X]]$ to $\CC_0([-a,a])$. He also mentioned that his proof can be easily adapted to treat the case of so-called Seleznev universal formal complex power series \cite{Seleznev1951}. From this point on, let us introduce some notations.
	
	Throughout the paper, $H(\D)$ will stand for the Fréchet space of all holomorphic functions on $\D$, endowed with the local uniform topology. For a compact set $K$ in $\C$, we will let $\CC(K)$ denote the Banach space of all continuous functions on $K$, endowed with the uniform topology, and the notation $A(K)$ will refer to the closed subspace of $\CC(K)$ consisting of these functions that are holomorphic in the interior of $K$. Given a compact set $K\subset \C$, let $S_n^{(K)}$ denote the $n$-th partial sum operator, defined by
	\begin{equation}\label{def-SnK}
		S_n^{K}:\left\{\begin{array}{lll}
			H(\D) & \to & A(K)\\
			f=\sum _{k=0}^\infty a_kz^k & \mapsto & \sum_{k=0}^n a_kz^k.
		\end{array}\right.
	\end{equation}
	By a slight abuse of notation, we will use $S_n$ for the partial sums of a function without specifying a compact set, as well as for the partial sums of any power series, not necessarily ones that define holomorphic functions on $\D$. 
	
	A function $f\in H(\D)$ that is universal for $(S_n^{K})_n$ will be called a \emph{$K$-universal Taylor series}. A Baire argument can be used to prove the existence of a dense $G_{\delta}$-subset of $H(\D)$, whose elements are $K$-universal Taylor series, for every $K\subset \C\setminus \D$ with connected complement \cite{Nestoridis1996}. We call these functions \emph{universal Taylor series}. These objects have been intensively investigated for 30 years; see for example \cite{MelasNestoridisPapadoperakis1997,MelasNestoridis2001,MullerVlachouYavrian2006,Gardiner2014,GardinerKhavinson2014}.
	
	Back to our motivation, the result of Mouze mentioned above can be formulated for $(S_n^{K})_n$ for specific compact sets $K$, as soon as one replaces, in \eqref{def-SnK},  the space $H(\D)$ by the Fréchet space of all formal complex power series $\C[[X]]$:
	\begin{theorem}[Mouze \cite{Mouze2020}]There exist $f_1$ and $f_2$ in $\C[[X]]$ such that, for any compact set $K\subset \{z\in \C:\, |z-1|<1\}$, the set
		\[
		\{S_n^{K}(f_1):\, n\in \N\}\cup \{S_n^{K}(f_2):\, n\in \N\}
		\]
		is dense in $A(K)$, while neither $f_1$ nor $f_2$ is universal for $(S_n^{K})_n$.
	\end{theorem}
	
	It is inherent in Mouze's proof that the two constructed functions $f_1$ and $f_2$ have zero radius of convergence. In particular, it does not tell us whether Theorems \ref{thmCP} and \ref{thmBF} hold, or not, for the sequence $(S_n^{K})_n$ as defined in \eqref{def-SnK}, for some non-trivial $K$. By ``non-trivial'', we mean with non-empty interior and with connected complement. Let us say that the case where $K$ is a singleton is not difficult to treat, see Remark 3 in \cite{Mouze2020}. The aim of this note is to fill this gap by complementing Mouze's proof in adding a key ingredient, namely approximation by incomplete polynomials. In this paper, the terminology \textit{incomplete polynomial} will refer to polynomials of the form $\sum_{k=\lfloor n/\tau \rfloor}^na_k z^k$, where $\tau >1$ is a fixed real number, and $n\in \N$. In general, $\tau$ can be allowed to vary with $n$. Approximation by incomplete polynomials was firstly investigated by Lorentz \cite{LorentzBook1977,LorentzBook1980}, and is a subject of interest for itself and for its numerous applications, especially on compact subsets of the real line \cite{Saff1983,SaffUllmanVarga1980}. In the complex plane, the problem has been investigated from the point of view of weighted polynomial approximation, see \cite{MhaskarSaff1985,SaffTotikBookLogpotextfields}. In particular, it was successfully addressed, to some extent, by Pritsker and Varga in \cite{PritskerVarga1998}. One of their results will be the key tool to prove the main result of this note, that can be formulated as follows.
	
	\begin{maintheorem}\label{mainthm}There exists a function $f\in H(\D)$ and a non-empty open disc $D\subset \C\setminus \D$ such that the set
		\[
		\{S_n^{\overline{D}}(f):\,n\in \N\}\cup \{S_n^{\overline{D}}(-f):\,n\in \N\}
		\]
		is dense in $\CC(\overline{D})$, while $f$ (and hence $-f$ too) is not a $\overline{D}$-universal Taylor series.
	\end{maintheorem}
	
	Thus, the sequence of operators $(S_n^{K})_n$ do not satisfy the property described in Theorems \ref{thmCP}, \ref{thmBF} and \ref{thmLM} for some non-trivial compact sets $K$. Subsequently, we mention that the same strategy of proof yields the existence of some non-trivial compact sets $K$ and some power series $\sum_k a_k z^k$ with positive radius of convergence, that define $K$-universal Taylor series, such that the set $\{a_k:\,k\in \N\}$ is not dense in $\C$. We mention that the corresponding problem for universal Taylor series is an open question; that is, if $f=\sum_{k}a_kz^k\in H(\D)$ defines a universal Taylor series, is it always true that $\{a_k:\,k\in \N\}$ is dense in $\C$? Our result provides us with a ``local'' counterexample to this problem.
	
	It is still an open question whether the latter and our main theorem hold true for any fixed compact set $K\subset \C\setminus \D$, with connected complement. Let us say that the previous was completely solved for universal sequences of dilations on $H(\D)$ \cite{CharpentierMouze2023}, which is related to the recently studied notion of Abel universal functions \cite{Charpentier2020,CharpentierManolakiMaronikolakis2023,CharpentierManolakiMaronikolakis2023_2,Maronikolakis2022}.
	
	Let us also mention that we further developed incomplete polynomial approximation and then applied our results to study frequently universal Taylor series in \cite{CharpentierMaronikolakis2025-1}.
	
	The paper is organised as follows. In the next section, preliminaries on weighted polynomial approximation are presented. Section \ref{sec-main} is devoted to the proof of the main results. In the expository Section \ref{sec-examples-compact}, explicit calculations are made of quantities inherent to our approach and to the potential theoretical tools used.
	
	\section{Weighted polynomial approximation - a result of Pritsker and Varga}\label{wpa}
	Let $G$ be a bounded domain in $\C$ and $W$ a function holomorphic in $G$ that does not vanish in $G$. Following \cite{PritskerVarga1998}, we say that $(G,W)$ is a pair of approximation if every function $f$ holomorphic in $G$ can be locally uniformly approximated on $G$ by the sequence $(W^nP_n)_n$ for some sequence $(P_n)_n$ of polynomials with $\text{deg}(P_n)\leq n$ , $n\in\N$. By $\overline{\C}$ we denote the Riemann sphere $\C\cup\{\infty\}$. Given an open set $\Omega\subset \overline{\C}$ and $z\in  \Omega$, we denote by $\omega(z,\cdot,\Omega)$ the harmonic measure with respect to $\Omega$, at $z\in \Omega$ \cite{Ransford1995}.
	
	\begin{theorem}[Theorem 2.2 in \cite{PritskerVarga1998}]\label{thm-PV}Let $G$ be a simply connected bounded domain contained in $\C\setminus (-\infty,0]$ and $W(z)=z^{\alpha}$, $z\in G$, for some $\alpha >0$ (we choose the principal branch of the logarithm). Then $(G,z^{\alpha})$ is a pair of approximation if and only if
		\begin{equation}\label{eq-PV}
			\mu:=(1+\alpha)\omega(\infty,\cdot,\overline{\C}\setminus \overline{G}) - \alpha\omega(0,\cdot,\overline{\C}\setminus \overline{G})
		\end{equation}
		is a positive measure. 
	\end{theorem}
	
	For more general weights, see Theorem 1.1 in \cite{PritskerVarga1998}. Having in mind that harmonic measures are conformally invariant and coincide with the Poisson measure if $\Omega=\D$, condition \eqref{eq-PV} is rather easy to check for various simple domains $G$ (see examples below). In practice, the following reformulation of Theorem \ref{thm-PV} can be useful, for computations, see Section \ref{sec-examples-compact}.
	
	\begin{cor}\label{poisson} Let $G$ be a Jordan domain contained in $\C\setminus (-\infty,0]$ whose boundary is an analytic curve, and let $\phi$ be a conformal homeomorphism from $\overline{\C}\setminus G$ to $\overline{\D}$. Let also $\alpha>0$. Then, $(G,z^{\alpha})$ is a pair of approximation if and only if
		\begin{equation}\label{eq-PV2}
			(1+\alpha)P(\phi(\infty),\phi(\zeta))-\alpha P(\phi(0),\phi(\zeta))\geq0,\quad\zeta\in\partial G,
		\end{equation}
		where $P:\D\times\T\to\R$ denotes the Poisson kernel $P(z,\zeta)=\frac{1}{2\pi}\frac{1-|z|^2}{|\zeta-z|^2}$, for $z\in\D$ and $\zeta\in\T$.
	\end{cor}
	
	Note that the existence of $\phi$ is ensured by the Osgood-Carathéodory Theorem. Moreover, since $\partial G$ is an analytic curve, $\phi$ is conformal in a neighbourhood of $\overline{\C}\setminus G$ and so in particular $\phi'$ is well defined and zero-free on $\partial G$ (see Proposition 3.1 in \cite{Pommerenke1992}).
	\begin{proof}
		First, we note the following regarding the harmonic measure of $\overline{\C}\setminus \overline{G}$:
		$$\omega(z,B,\overline{\C}\setminus \overline{G})=\omega(\phi(z),\phi(B),\D)$$
		for $z\in\overline{\C}\setminus\overline{G}$ and $B$ a Borel subset of $\partial G$ (see Theorem 4.3.8 in \cite{Ransford1995}). Thus, we get that
		$$\omega(z,B,\overline{\C}\setminus \overline{G})=\int_{\phi(B)}P(\phi(z),\zeta)|d\zeta|=\int_{B}P(\phi(z),\phi(\xi))|\phi^{'}(\xi)||d\xi|$$
		which means that
		$$d\omega(z,\cdot,\overline{\C}\setminus \overline{G})=P(\phi(z),\phi(\zeta))|\phi^{'}(\zeta)||d\zeta|$$
		where $|d\zeta|$ denotes the arclength measure on $\T$. Therefore, \eqref{eq-PV} is satisfied if and only if
		$$(1+\alpha)P(\phi(\infty),\phi(\zeta))|\phi^{'}(\zeta)|-\alpha P(\phi(0),\phi(\zeta))|\phi^{'}(\zeta)|\geq0\text{ for any }\zeta\in\partial G$$
		which is clearly equivalent to \eqref{eq-PV2}.
	\end{proof}
	
	Actually, the strong properties of positive harmonic functions imply that for any $G$ as in Theorem \ref{thm-PV}, there is always $\alpha >0$ such that $(G,z^{\alpha})$ to be a pair of approximation:
	
	\begin{prop}\label{prop-existence-alpha-PA}Let $G$ be a simply connected bounded domain contained in $\C\setminus (-\infty,0]$. There exists $\alpha >0$ such that $(G,z^{\alpha})$ is a pair of approximation.
	\end{prop}
	
	\begin{proof}Let us recall that for any Borel set $B\subset \partial G$, the map $z\mapsto \omega(z,B,\overline{\C}\setminus \overline{G})$, $z\in \overline{\C}\setminus \overline{G}$, is a positive harmonic function. By Harnack's inequality, there exists $C>1$ such that $\omega(0,B,\overline{\C}\setminus \overline{G}) < C \omega(\infty,B,\overline{\C}\setminus \overline{G})$ for any $B$. Thus, choosing $\alpha < \frac{1}{C-1}$, we get
		\[
		(1+\alpha)\omega(\infty,B,\overline{\C}\setminus \overline{G}) - \alpha\omega(0,B,\overline{\C}\setminus \overline{G})\geq (1-\alpha(C-1))\omega(\infty,B,\overline{\C}\setminus \overline{G}) >0,
		\]
		for any Borel set $B\subset \partial G$.
	\end{proof}
	Next, we prove another property that is of independent interest. Let $H_1\subset H_2$ be two simply connected bounded domain contained in $\C\setminus (-\infty,0]$. It is clear by the definition that for any $\alpha>0$ such that $(H_2,z^\alpha)$ is a pair of approximation, then $(H_1,z^\alpha)$ is also a pair of approximation. It is also clear by Theorem \ref{thm-PV} that, if $\beta < \alpha$ and $(G,z^{\alpha})$ is a pair of approximation, then $(G,z^{\beta})$ is a pair of approximation too. The next proposition essentially tells us that if $\overline{H_1}\subset H_2$ then there are strictly more choices of $\alpha>0$ for which $(H_1,z^\alpha)$ is a pair of approximation compared to $(H_2,z^\alpha)$.
	\begin{prop}\label{strict_inc}
		Let $H_1,H_2$ be two simply connected bounded domain contained in $\C\setminus (-\infty,0]$ such that $\overline{H_1}\subset H_2$. Then, there exists $C>0$ such that, for any $\alpha>0$, if $(H_2,z^{\alpha})$ is a pair of approximation, then $(H_1,z^{\alpha+C})$ is a pair of approximation.
	\end{prop}
	\begin{proof}
		We will use the notion of the balayage measure (we refer to Chapter IV of \cite{Landkof1972} for the definition and details). Let $H$ be a domain in $\overline{\C}$ whose boundary is not a polar set and $z\in H$. It is known that the balayage measure of the Dirac measure at $z$ out of $H$ is equal to the harmonic measure $\omega(z,\cdot,H)$ (see p.222 in \cite{Landkof1972}).
		
		Let $n\in\N$ and $z\in\overline{\C}\setminus\overline{H_2}$. Then, by the previous, the balayage measure of the Dirac measure at $z$ out of $\overline{\C}\setminus\overline{H_2}$ (respectively $\overline{\C}\setminus\overline{H_1}$) is equal to the harmonic measure $\omega(z,\cdot,\overline{\C}\setminus\overline{H_2})$ (respectively $\omega(z,\cdot,\overline{\C}\setminus\overline{H_1})$). Moreover, by the uniqueness of the balayage measure, the measure $\omega(z,\cdot,\overline{\C}\setminus\overline{H_1})$ is also equal to the balayage measure of $\omega(z,\cdot,\overline{\C}\setminus\overline{H_2})$ out of $\overline{\C}\setminus\overline{H_1}$ (see p.208 of \cite{Landkof1972}). Let $B$ be a Borel subset of $\partial H_1$. Combining the previous facts with equation (4.1.10) on p.221 in \cite{Landkof1972}, we get
		\begin{equation}
			\omega(z,B,\overline{\C}\setminus\overline{H_1})=\int_{\partial H_2}\omega(\zeta,B,\overline{\C}\setminus\overline{H_1})d\omega(z,\zeta,\overline{\C}\setminus\overline{H_2}).
		\end{equation}
		Let $\alpha>0$ such that $(H_2,z^{\alpha})$ is a pair of approximation. Setting $\mu:=	(1+\alpha)\omega(\infty,\cdot,\overline{\C}\setminus\overline{H_2}) - \alpha\omega(0,\cdot,\overline{\C}\setminus\overline{H_2})$, we get
		\begin{equation}\label{harm_eq_2}
			(1+\alpha)\omega(\infty,B,\overline{\C}\setminus\overline{H_1}) - \alpha\omega(0,B,\overline{\C}\setminus\overline{H_1})=\int_{\partial H_2}\omega(\zeta,B,\overline{\C}\setminus\overline{H_1})d\mu(\zeta).
		\end{equation}
		Since the set $\partial H_2$ is a compact subset of $\overline{\C}\setminus\overline{H_1}$, by Harnack's inequalities, there exists $C>0$, independent of $B$, such that $$\omega(\zeta,B,\overline{\C}\setminus\overline{H_1})\geq C\omega(0,B,\overline{\C}\setminus \overline{H_1})$$
		for any $\zeta\in\partial H_2$. This gives us
		$$\int_{\partial H_2}\omega(\zeta,B,\overline{\C}\setminus\overline{H_1})d\mu(\zeta)\geq C\omega(0,B,\overline{\C}\setminus\overline{H_1})$$
		using the fact that $\mu$ is a (positive) probability measure. Combining with \eqref{harm_eq_2}, since $C$ is independent of $B$,
		\begin{equation}
			(1+\alpha)\omega(\infty,B,\overline{\C}\setminus\overline{H_1}) - (\alpha+C)\omega(0,B,\overline{\C}\setminus\overline{H_1})\geq0,
		\end{equation}
		for any Borel subset $B$ of $\partial H_1$.
		It follows that the measure
		$$(1+\alpha+C)\omega(\infty,\cdot,\overline{\C}\setminus\overline{H_1}) - (\alpha+C)\omega(0,\cdot,\overline{\C}\setminus\overline{H_1})$$ is positive. Finally, we note that $C$ is also independent of $\alpha$, which completes the proof.
	\end{proof}
	To our knowledge, there is no characterization of the pairs $(K,z^{\alpha})$, where $K$ is a compact set in $\C\setminus (-\infty,0]$ and $\alpha >0$, that have the property that, for any function $\varphi$ continuous in $K$ and holomorphic in its interior, there exist weighted holomorphic polynomials of the form $z^{\alpha n}P_n$, with $\text{deg}(P_n)\leq n$, such that $z^{\alpha n}P_n \to \varphi$ uniformly on $K$ (see Section 5 in \cite{PritskerVarga1998}). However, one can exploit the previous proposition to get the following:
	\begin{prop}\label{prop-alphaK}
		Let $K\subset\C\setminus (-\infty,0]$ be a compact set with connected complement. Then there exists $\alpha_K\in(0,\infty)$ such that:
		\begin{itemize}
			\item if $0<\alpha<\alpha_K$, then there exists a simply connected bounded domain $G\subset\C\setminus (-\infty,0]$ that contains $K$ such that $(G,z^{\alpha})$ is a pair of approximation;
			\item if $\alpha\geq\alpha_K$, then there exists no simply connected bounded domain $G\subset\C\setminus (-\infty,0]$ that contains $K$ such that $(G,z^{\alpha})$ is a pair of approximation.
		\end{itemize}
	\end{prop}
	\begin{proof}
		Let $(G_n)_n$ be a sequence of simply connected bounded domains in $\C\setminus (-\infty,0]$ with $\overline{G_{n+1}}\subseteq G_n$ and $\cap_{n\in\N}G_n=K$.
		The sets $G_n$ are regular for the Dirichlet problem, see Theorem 4.2.1 in \cite{Ransford1995}. By Theorem \ref{thm-PV} and Proposition \ref{prop-existence-alpha-PA}, for every $n\in \N$, there exists $\alpha_{K,n}>0$ such that $(G_n,z^{\alpha})$ is a pair of approximation if and only if $\alpha\leq\alpha_{K,n}$.
		
		By Proposition \ref{strict_inc}, it is clear that the sequence $(\alpha_{K,n})_n$ is strictly increasing, so it has a limit which we denote by $\alpha_K$ (allowing $\alpha_K=\infty$).
		
		Assume that $\alpha_K<\infty$ (the case $\alpha_K=\infty$ is treated in a similar way). For $0<\alpha<\alpha_K$, there exists $n\in\N$ such that $\alpha_{K,n}>\alpha$. Then $(G_n,z^\alpha)$ is a pair of approximation.
		
		Now suppose for the sake of contradiction that there exists a simply connected bounded domain $G\subset \C\setminus (-\infty,0]$ that contains $K$ such that $(G,z^{\alpha_K})$ is a pair of approximation. Let $n\in\N$ be such that $G_n\subseteq G$. Then $(G_n,z^{\alpha_K})$ is a pair of approximation, hence $\alpha_K\leq\alpha_{K,n}$. But this contradicts the fact that the sequence $(\alpha_{K,n})_n$ is strictly increasing. It follows that $(G,z^{\alpha})$ is a pair of approximation for no $\alpha \geq \alpha_K$.
	\end{proof}

	\section{Main results}\label{sec-main}
	
	This section is devoted to the statements and the proofs of our main results. As mentioned before, the key-ingredient is Theorem \ref{thm-PV}. We will actually derive from it a lemma of approximation by polynomials with a specific control on the coefficients and the partial sums. This result will be central in the forthcoming proofs. The statement will involve several parameters depending on a fixed compact set $K$ and a fixed positive rational number $\alpha$. We will give explicit values of these parameters in order to be as optimal as the method allows us to be. If $K\subset \C\setminus \D$ is a compact set, let 
	\begin{equation}\label{defi-MK}
		M_K:=\sup\left\{e^{g_{\overline{\C}\setminus K}(z,\infty)}:\, z\in \overline{\D}\right\}
	\end{equation}
	where $g_{\overline{\C}\setminus K}(\cdot,\infty)$ denotes the Green function with pole at $\infty$ of the open set $\overline{\C}\setminus K$. Let us also fix a positive (irreducible) rational number $\alpha=\sigma/\tau>0$. We introduce the following notations:
	\begin{notations}\label{notation-rKalpha}{\rm With the above notations, we set:
			\begin{itemize}
				\item $r(K,\alpha):=\sup\left\{r\in [0,1):\,r^{\sigma}\left(1+r\right)^{\tau}<M_K^{-\tau}\right\}$;
				\item $K(\alpha):=\left\{z\in K:\,\left\vert z\right\vert^{\sigma}\left\vert1-z\right\vert^{\tau}<M_K^{-\tau}\right\}$;
			\end{itemize}
		}
	\end{notations}
	
	\noindent{}Observe that $r(K ,\tau)>0$ and that, if $K$ contains $1$ in its interior, then $K(\alpha)$ contains some non-empty open disc centered at $1$.
	
	Given $C>0$, $\sigma,\tau\in \N\setminus \{0\}$, and $n\in \N$, let us now define the polynomial $\Pi_{n,C}^{\sigma,\tau}$ by
	\begin{equation}\label{def-Pi-n}
		\Pi_{n,C}^{\sigma,\tau}(z)=C^nz^{\sigma n}\left(1-z\right)^{\tau n}.
	\end{equation}
	\noindent{}The introduction of these polynomials for our purpose, and the following fact, are inspired by the proof of Proposition 2.2 in \cite{Mouze2020}:
	\begin{fact}\label{fact-Pin}With the above notations,
		\begin{enumerate}[(a)]
			\item one can write $\Pi_{n,C}^{\sigma,\tau}(z)=\sum_{k=\sigma n}^{(\sigma +\tau)n}b_kz^k$ with $|b_k|\geq C^n$, $\sigma n\leq k \leq (\sigma +\tau)n$;
			\item $\Pi_{n,C}^{\sigma,\tau}(1)=0$ and, for any $\sigma n\leq j\leq (\sigma + \tau)n -1$, $|S_j(\Pi_{n,C}^{\sigma,\tau})(1)|=|\Re \left(S_j(\Pi_{n,C}^{\sigma,\tau})(1)\right)|\geq C^n$.
		\end{enumerate}
	\end{fact}
	
	\begin{proof}(a) By definition, a direct computation yields
		\[
		\Pi_{n,C}^{\sigma,\tau}(z) = C^n \sum_{k=\sigma n}^{(\sigma + \tau)n}\binom{\tau n}{k-\sigma n}(-1)^{k-\sigma n}z^k.
		\]
		Setting $b_k=C^n\binom{\tau n}{k-\sigma n}(-1)^{k-\sigma n}$, we get the second part of (a).
		
		(b) The first assertion is obvious. For the second, let $\sigma n\leq j\leq (\sigma + \tau)n -1$. The previous expression leads to
		\begin{eqnarray*}
			S_j(\Pi_{n,C}^{\sigma,\tau})(1) & = & \Re\left(S_j(\Pi_{n,C}^{\sigma,\tau})(1)\right)\\
			& = & C^n\sum_{k=0}^{j-\sigma n}\binom{\tau n}{k}(-1)^k\\
			& = & C^n (-1)^{j-\sigma n}\binom{\tau n-1}{j-\sigma n},
		\end{eqnarray*}
		where the last equality, which is rather standard, can be checked by induction.
	\end{proof}

	We are ready to state and prove our lemma. 	From now on, if $P(z)=\sum_{k=0}^{n}a_kz^k$ is a non-zero polynomial then we denote by $\text{deg}(P)$ its degree and by $\text{val}(P)$ its valuation, that is $\text{val}(P):=\min\{k\in\N:a_k\neq0\}$.  We recall that, by Proposition \ref{prop-alphaK}, for any compact set $K\subset \C\setminus (-\infty,0]$ and any $\alpha \in (0,\alpha_K)$, there exists a simply connected bounded domain $G\subset \C\setminus (-\infty,0]$ containing $K$, such that $(G,z^{\alpha})$ is a pair of approximation.
	
	\begin{lemma}\label{lemma-approx-PV-control-partial-sums}Let $K\subset \C\setminus \D$ be a compact set, with connected complement. We assume that $1 \in K$ and that $K\cap (-\infty,0]=\emptyset$. Then, for any  $\alpha \in (0,\alpha_K)$, $\varepsilon>0$, $\vp \in A(K)$, $N\in \N$, $B\geq 0$, $r\in(0,r(K,\alpha))$ and any compact set $L\subset K(\alpha)$, there exists a polynomial $P=\sum_{k=v}^da_kz^k$ with $d>v\geq N$ such that:
		\begin{enumerate}
			\item $\sup_{z\in L}|P(z)-\vp(z)|< \varepsilon$;
			\item $\sup_{z\in \overline{D(0,r)}}|P(z)|<\varepsilon$;
			\item $\left|\Re(\sum_{k=N}^ja_k)\right|\geq B$, for any $v\leq j\leq d-1$;
			\item $\left|a_k\right|\geq B$, for any $v\leq k\leq d$;
		\end{enumerate}
	\end{lemma}
	
	\begin{proof}The idea is to construct $P$ as a sum of two polynomials: the first one, given by Theorem \ref{thm-PV}, will guarantee the desired simultaneous approximation (items (1) and (2)); The second one will suitably perturbate the first one in order to get (3) and (4).
		
		Let $\alpha$, $K$, $\vp$, and $\varepsilon>0$ be fixed as in the statement. Let $G$ be a simply connected bounded domain in $\C\setminus (-\infty,0]$, containing $K$, such that $(G,z^{\alpha})$ is a pair of approximation (Proposition \ref{prop-alphaK}). Upon adding a small segment to $K$, we may and shall assume that $K$ is non-polar. By Mergelyan's theorem, we may and shall assume that $\vp$ is a polynomial. Let us write $\alpha=\sigma/\tau$, with $\sigma$ and $\tau$ positive integers, and let us fix $r\in(0,r(K,\alpha))$ and a compact set $L\subset K(\alpha)$. By the definitions of $r(K,\alpha)$ and $K(\alpha)$, there exists $C>M_K^{\tau}$ such that
		\begin{equation}\label{choice-of-C}
			Cr^{\sigma}\left(1+r\right)^{\tau}<1 \quad \text{and} \quad \sup_{z\in L}C\left\vert z\right\vert^{\sigma}\left\vert1-z\right\vert^{\tau}<1.
		\end{equation}
		
		By Theorem \ref{thm-PV}, there exists a sequence $(Q_n)$ of polynomials, with $\text{deg}(Q_n)\leq n$, such that $\|z^{n\alpha}Q_n - \vp\|_{K}$ goes to $0$ as $n\to \infty$. Using $\alpha=\sigma/\tau$,  we get that the sequence $(z^{n\sigma}Q_{n\tau})_n$ is also convergent to $\vp$ uniformly on $K$. Let us fix $n_0\in \N$ such that
		\begin{equation}\label{ineq-first-Merg-on-K}
			\|z^{n\sigma}Q_{n\tau}-\vp\|_K < \varepsilon/2,\quad n\geq n_0.
		\end{equation}
		Fix any $n\geq n_0$. Since $K$ is non-polar, Bernstein's inequality (see \cite[Theorem 	5.5.7]{Ransford1995}) and the fact that $K\subset\C\setminus \D$ yield
		\[
		|Q_{n\tau}(z)| \leq \|Q_{n\tau}\|_Ke^{n\tau g_{\overline{\C}\setminus K}(z,\infty)}\leq \|z^{n\sigma}Q_{n\tau}\|_KM_K^{n\tau},\quad z\in \overline{\D}.
		\]
		Thus
		\begin{equation}\label{ineq-Qnl-Bernstein}
			|z^{n\sigma}Q_{n\tau}(z)| \leq \left(|z|^{\sigma}M_K^{\tau}\right)^{n}\|z^{n\sigma}Q_{n\tau}\|_K,\quad z\in \overline{\D}.
		\end{equation}
		Note that the first inequality of \eqref{choice-of-C} implies $r^{\sigma}M_K^{\tau}<1$, hence $\left(|z|^{\sigma}M_K^{\tau}\right)^{n} \to 0$ as $n\to \infty$, uniformly for $z\in \overline{D(0,r)}$. Since  the sequence $(\|z^{n\sigma}Q_{n\tau}\|_K)_{n\geq n_0}$ is bounded (by \eqref{ineq-first-Merg-on-K}), there exists $n_1\geq n_0$, such that for any $n\geq n_1$,
		\begin{equation}\label{ineq-Qnl}
			\|z^{n\sigma}Q_{n\tau}-\vp\|_K < \varepsilon/2\quad \text{and}\quad \|z^{n\sigma}Q_{n\tau}\|_{\overline{D(0,r)}}<\varepsilon/2.
		\end{equation}
		
		We shall now perturbate the polynomial $z^{n\sigma}Q_{n\tau}$. For $n\geq n_1$ and $C$ as above, let us set
		\[
		P_n(z)=z^{n\sigma}Q_{n\tau}+\Pi_{n,C}^{\sigma,\tau}(z);
		\]
		see \eqref{def-Pi-n} for the definition of $\Pi_{n,C}^{\sigma,\tau}$. By the definitions of $r(K,\alpha)$ and $K(\alpha)$, and the choice of $C$, it is clear that there exists $n_2\geq n_1$ such that for any $n\geq n_2$,
		\[
		\|P_n(z)-\vp\|_L < \varepsilon/2\quad \text{and}\quad \|P_n\|_{\overline{D(0,r)}}<\varepsilon/2,
		\]
		that is, $P_n$ satisfies (1) and (2) for any $n\geq n_2$.
		
		It remains to check that $P_n$ satisfies (3) and (4). To do so, let us estimate the coefficients of $z^{n\sigma}Q_{n\tau}$, and its partial sums at $1$, for $n\geq n_2$.  Let us write $z^{n\sigma}Q_{n\tau}=\sum_{k=n\sigma}^{n(\sigma+\tau)}c_kz^k$ with $c_k$ in $\C$.  Using that $K\subset \C\setminus \D$ together with \eqref{ineq-first-Merg-on-K} and \eqref{ineq-Qnl-Bernstein}, Cauchy's inequalities yield, for any $n\sigma\leq k\leq n(\sigma+\tau)$,
		\begin{equation}\label{eq-est-coef-ck}
			|c_k|\leq  M_K^{\tau n}\left(\|\vp\|_K+\varepsilon/2\right),
		\end{equation}
		and, for any $n\sigma\leq j<n(\sigma+\tau)$,
		\begin{equation}\label{ineq-partial-sum-znsQnl}
			|S_j(z^{n\tau}Q_{n\sigma})(1)|\leq n\tau M_K^{\tau n}\left(\|\vp\|_K+\varepsilon/2\right).
		\end{equation}
		Let us denote by $a_{k,n}$ the $k$th Taylor coefficient of $P_n$. By construction, Fact \eqref{fact-Pin} and the inequalities \eqref{eq-est-coef-ck} and \eqref{ineq-partial-sum-znsQnl}, we get for any $n\sigma \leq k\leq n(\sigma +\tau)$,
		\[
		|a_{k,n}|\geq C^n - M_K^{\tau n}\left(\|\vp\|_K+\varepsilon/2\right),
		\]
		and for any $n\sigma \leq j< n(\sigma +\tau)$,
		\[
		|S_j(P_n)(1)|\geq C^n - n\tau M_K^{\tau n}\left(\|\vp\|_K+\varepsilon/2\right).
		\]
		Finally, using $C>M_K^{\tau}$ , there exists $n\geq n_2$ such that $P:=P_n$ has all the desired properties.
	\end{proof}
	
	\medskip
	
	We shall now derive from this lemma answers to the motivational problems mentioned in the introduction. First of all, we will use it to construct a function in $H(\D)$ whose partial sums are dense in a half-space of $A(K)$, but not in the whole space $A(K)$. Given a real number $s$, we shall denote by $\C_{s}^+$ the half-space $\{z\in \C:\,\Re(z)\geq s\}$ and by $\C_{s}^-$ the half-space $\{z\in \C:\,\Re(z)\leq s\}$. We keep on using Notations \ref{notation-rKalpha}.
	
	\begin{theorem}\label{thm-Cos-Peris-restricted}Let $K\subset \C\setminus \D$ be a compact set, with connected complement and let $\alpha \in (0,\alpha_K)$. We assume that $1 \in K$ and that $K\cap (-\infty,0]=\emptyset$. Then there exists a function $f\in H(D(0,r(K,\alpha)))$ such that the following properties hold true:
		\begin{enumerate}
			\item \label{item-1-thm-CP}for any $\vp \in A(K)$ with $\vp(1)\in\C_{0}^+$ and any compact subset $L$ of $K(\alpha)$ with connected complement, there exists a sequence $(\lambda_n)_n$ of integers such that
			\[
			S_{\lambda_n}(f) \to \vp \quad \text{uniformly on }L\text{ as }n\to \infty;
			\]
			\item \label{item-2-thm-CP}for any $n\in \N$, $S_n(f)(1)\in \C_{-1}^-\cup \C_{0}^+$.
		\end{enumerate}
	\end{theorem}
	
	\begin{proof}Let $\alpha$ and $K$ be fixed as in the statement. Let $(\vp_n)_n$ be a sequence contained and dense in the half-space $\{\vp \in A(K):\,\vp(1)\in \C_0^+\}$ of $A(K)$. Without loss of generality, we may and shall also assume that $\Re{(\vp_n(1))}>0$. Let us also fix two sequences of positive real numbers $(\varepsilon_n)_n$ and $(s_n)_n$ such that $(\varepsilon_n)_n$ decreases to $0$ and satisfies $0<\varepsilon_n < \Re(\vp_n(1))$, and $(s_n)_n$ increases to $r(K,\alpha)$. Let also $(L_n)_n$ be an increasing exhaustion of $K(\alpha)$ by compact sets, such that
		\[
		\sup_{z\in L_n}\left\vert z\right\vert^{\sigma}\left\vert1-z\right\vert^{\tau}<\left(1-\frac{1}{n}\right)M_K^{-\tau},\, n\in \N.
		\]
		
		Let us now build by induction polynomials $P_n$, $n\in \N$, with the help of Lemma \ref{lemma-approx-PV-control-partial-sums}. We set $P_0=0$ and suppose that $P_1,\ldots P_{n-1}$ have been built. We apply Lemma \ref{lemma-approx-PV-control-partial-sums}, with $\vp=\vp_n-\sum_{k=0}^{n-1}P_k$, $\varepsilon=\varepsilon_n$, $N=\text{deg}(P_{n-1})+1$, $r=s_n$, $L=L_n$ and $B=|\Re(\sum_{k=0}^{n-1}P_k(1))|+1$. It directly follows that, for any $n\geq 1$,
		\begin{enumerate}[(a)]
			\item $\sup_{z\in L_n}|\sum_{k=0}^{n}P_k(z) - \vp_n(z)|< \varepsilon_n$;
			\item $\sup_{z\in \overline{D(0,s_n)}}|P_n(z)|<\varepsilon_n$;
			\item $S_j(\sum_{k=0}^nP_k)(1) \in \C_{-1}^-\cup \C_{1}^+$ for any $N\leq j \leq \text{deg}(P_n)-1$.
		\end{enumerate}
		We set $g=\sum_{n\geq 0}P_n$.
		Since $s_n\to  r(K,\alpha)$ and $\varepsilon_n\to 0$, it comes from (b) that $g$ is holomorphic in $D(0,r(K,\alpha))$. Moreover, 
		note that for any compact set $L\subset K(\alpha)$, there exists $n\geq 1$ such that $L\subset L_n$. Then
		we deduce from (a) that $f$ satisfies (1). To see that (2) holds as well, it suffices to notice that, by (c), $S_j(f)(1)  \in \C_{-1}^-\cup \C_{1}^+$ for any $\text{val}(P_n)\leq j\leq \text{deg}(P_n)-1$ and any $n\in \N$, while (a) and the assumption on $(\varepsilon_n)_n$ give $S_j(f)(1) \in \C_{0}^+$ for any $\text{deg}(P_n)\leq j \leq \text{val}(P_{n+1})-1$ and any $n\in \N$.
	\end{proof}
	
	It is clear that if $f$ is given by Theorem \ref{thm-Cos-Peris-restricted}, then $-f$ satisfies the same conclusion, upon replacing in (1) the set $\C_{0}^+$ by $\C_{0}^-$ and in (2) the set $\C_{-1}^-\cup \C_{0}^+$ by $\C_{1}^+\cup \C_{0}^-$. Thus, since
	\[
	A(L)=\{\vp \in A(L):\, \vp(\zeta)\in \C_{0}^+\}\cup \{\vp \in A(L):\, \vp(\zeta)\in \C_{0}^-\},
	\]
	Theorem \ref{thm-Cos-Peris-restricted} implies  the following:
	
	\begin{cor}\label{cor-CP-corollary}With the notations and under the assumptions of Theorem \ref{thm-Cos-Peris-restricted}, there exists $f$ holomorphic in $D(0,r(K,\alpha))$ such that for any compact set $L\subset K(\alpha)$ with connected complement,
		\[
		A(L)=\overline{\{S_n^L(f):\,n\in \N\}\cup\{S_n^L(-f):\,n\in \N\}},
		\]
		while $\overline{\{S_n^L(f):\,n\in \N\}}\neq A(L)$.
	\end{cor}
	Since the map $\varphi \mapsto \varphi_r$, given by $\varphi_r(z) = \varphi(z/r)$ is an isometry from $A(K)$ to $A(K_r)$, where $K_r=\{rz\in \C:\, z\in K\}$, for any $r>0$, the previous result imply our main result. Indeed, upon composing $f$ from the right by the dilation $z\mapsto r(K,\alpha)z$, Corollary \ref{cor-CP-corollary} yields the following:
	
	\begin{cor}\label{cor-after-dilation}There exists a non-empty open disc $D \subset \C\setminus \D$ such that there exists $f \in H(\D)$, that is not in $\UU(\D,\overline{D})$, such that
		\[
		\overline{\{S_n^{\overline{D}}(f):\,n\in \N\}}\cup \overline{\{-S_n^{\overline{D}}(f):\,n\in \N\}}=A(\overline{D}).
		\]
	\end{cor}
	\begin{remark}{\rm \quad
			\begin{enumerate}
				\item It is clear that if $f$ has the property described in Corollary \ref{cor-after-dilation}, then for any $\zeta \in \T$ and any $r\geq 1$, the function $f(\frac{\zeta}{r}z)$ satisfies the same property with $D$ replaced by the disc $\{r\overline{\zeta}z:\,z\in D\}$.
				\item We should mention that, in the proof of Lemma \ref{lemma-approx-PV-control-partial-sums}, we may have replaced the main ingredient, that is Theorem \ref{thm-PV}, by Theorem C in \cite{CharpentierMaronikolakis2025-1}. The latter directly provides us with uniform approximation on $K\subset \C\setminus \overline{\D}$ and, simultaneously, uniform approximation of $0$ on $\overline{\D}$. This way, we would avoid the use of Bernstein's inequality. Subsequently, we would not have to dilate in order to get Corollary \ref{cor-after-dilation}. However, it is still not clear whether this approach would provide anything better here. Moreover, it would make us lose the possibility to find the range of the parameter $\alpha$, even for simple compact sets $K$; see the next section.
			\end{enumerate}
		}
	\end{remark}
	\medskip
	
	Let us note that in the proof of Theorem \ref{thm-Cos-Peris-restricted}, we do not use the assertion (4) of Lemma \ref{lemma-approx-PV-control-partial-sums}. But, if we repeat this proof identically, except that the sequence $(\varphi_n)_n$ is now the sequence of all polynomials with coefficients in $\Q+i\Q$ then, applying Lemma \ref{lemma-approx-PV-control-partial-sums} with $B=n$ at each step of the inductive construction, and then taking into account the assertion (4) of the lemma, we obtain the following result (details are left to the reader).
	
	\begin{theorem}\label{thm-coef}With the notations and under the assumptions of Theorem \ref{thm-Cos-Peris-restricted}, there exists $f=\sum_ka_kz^k$, holomorphic in $D(0,r(K,\alpha))$, satisfying $\displaystyle{\lim_{\substack{n\to \infty \\ a_n\neq 0}}|a_n|\to +\infty}$ as $n\to \infty$, such that for any compact set $L\subset K(\alpha)$ with connected complement,
		\[
		\overline{\{S_n^L(f):\,n\in \N\}}=A(L).
		\]
	\end{theorem}

	The results of this section leave some old questions open and raise some new relevant ones.
	
	\begin{questions}\quad{\rm 
			\begin{enumerate}
				\item Do theorems \ref{thm-Cos-Peris-restricted} and \ref{thm-coef} hold for a prescribed compact set in $\C\setminus \D$? It seems to us that our method and the use of weighted polynomial approximation - at least Theorem \ref{thm-PV}, cannot help on this problem.
				\item In Corollary \ref{cor-after-dilation}, we construct two functions $f,g \in H(\D)$, that are not in $\UU(\D,\overline{D})$, such that
				\[
				\overline{\{S_n^{\overline{D}}(f):\,n\in \N\}}\cup \overline{\{-S_n^{\overline{D}}(g):\,n\in \N\}}=A(\overline{D}).
				\]
				Of course, we actually get that $g=-f$ and so in particular $f$ and $g$ are linearly dependent. A natural question that arises is then the following: could we construct  two linearly independent functions $f,g \in H(\D)$ with the same properties as described above?
				
				We also notice that, for these two functions $f,g=-f$, we have that the sets $\{\lambda S_n^{\overline{D}}(f):\,n\in \N,\lambda\in\C\}$ and $\{\lambda S_n^{\overline{D}}(g):\,n\in \N,\lambda\in\C\}$ are dense in $A(\overline{D})$. Essentially, this tells us that $f,g$ are ``supercyclic'' for the sequence of partial sums $(S_n^{\overline{D}})_n$. Could we construct two functions $f,g$ such that the sets $\{\lambda S_n^{\overline{D}}(f):\,n\in \N,\lambda\in\C\}$ and $\{\lambda S_n^{\overline{D}}(g):\,n\in \N,\lambda\in\C\}$ are not dense in $A(\overline{D})$, but the outcome of Corollary \ref{cor-after-dilation} still holds?
		\end{enumerate}}
	\end{questions}
	\section{Estimates of $\alpha_K$ and $M_K$ for some simple compact sets $K$}\label{sec-examples-compact}
	We keep the notation of Section \ref{sec-main}. In Corollary \ref{cor-CP-corollary} and Theorem \ref{thm-coef}, we would like the quantity $r(K,\alpha)$ and the set $K(\alpha)$ to be as large as possible. Similarly, in Corollary \ref{cor-after-dilation}, it would be good to have some estimate on the size of the disc $D$. Despite the fact that these quantities basically depend only on $K$, it appears that the choice of $\alpha$ has a somewhat opposite influence on $r(K,\alpha)$ and $K(\alpha)$. Then, observing that, equivalently,
	\begin{multline*}
		r(K,\alpha)=\sup\left\{r\in [0,1):\,M_Kr^{\alpha}\left(1+r\right)<1\right\} \text{ and }K(\alpha)=\left\{z \in K:\, M_K\left\vert z\right\vert^{\alpha}\left\vert1-z\right\vert<1\right\},
	\end{multline*}
	it appears that the larger $\alpha$ is, the larger $r(K,\alpha)$ is, while the smaller (in terms of inclusion) $K(\alpha)$ is. Further, to get a better understanding about the relationship between $r(K,\alpha)$ and $K(\alpha)$, it is enough to get estimates on $M_K$ and to get the range of $\alpha$. To do the latter, we recall that, in the results of the previous section, it is a crucial assumption that $(G,z^{\alpha})$ is a pair of approximation for some domain $G$ containing $K$. Now, by Proposition \ref{prop-alphaK}, we know that $\alpha$ must be chosen in the interval $(0,\alpha_K)$.
	
	Therefore, understanding $r(K,\alpha)$ and $K(\alpha)$ eventually reduces to estimate $\alpha_K$ and $M_K=\sup\{e^{g_{\C\setminus K}(z,\infty)}:\, z\in \overline{\D}\}$. In general, this is quite difficult. Nevertheless, we shall explicitly compute these quantities for simple compact sets $K$ satisfying the above assumption, namely for intervals contained in the real line, and discs. For subarcs of the unit circle, we shall also compute $\alpha_K$, but we will only get an upper estimate for $M_K$.
	
	To calculate $\alpha_K$, we shall use Corollary \ref{poisson}. For $M_K$, we shall use the following result about estimating the Green function.
	\begin{theorem}[Theorem 1 in \cite{Solynin2014}]\label{Solynin}
		Let $K\subset \C$ be a continuum with connected complement, and let $\Omega(K)$ be the unbounded component of $\overline{\C}\setminus K$. Then, for any $z\in\Omega(K)$ we have that
		\begin{equation}\label{green_estimate}
			g_{\Omega(K)}(z,\infty)\leq\Phi\left(\frac{\textup{dist}(z,K)}{\textup{diam}(K)}\right)<2\sqrt{\frac{\textup{dist}(z,K)}{\textup{diam}(K)}}
		\end{equation}
		where $\Phi:[0,\infty)\to\R$ is the function $\Phi(x)=2\log\left(\sqrt{1+x}+\sqrt{x}\right)$.
		
		Moreover, equality occurs in the first inequality in \eqref{green_estimate} if and only if $K$ is a line segment and $z$ belongs to the line containing $K$.
	\end{theorem}
	
	\subsection{When $K$ is a closed disk in $\C\setminus \D$, tangent to $\D$ at $1$}
	The problem of weighted polynomial approximation for discs was completely solved in \cite{PritskerVarga1998}. More specifically, by Corollary 2.3 in \cite{PritskerVarga1998}, we immediately get the following:
	\begin{prop}Let $x_0>1$ and $K=\overline{D(x_0,x_0-1)}$. Then $\alpha_K=\frac{1}{2x_0-1}$.
	\end{prop}
	Using the conformal invariance of the Green function and the formula for the Green function of $\D$, that is $g_{\D}(z,w)=\log\left|\frac{1-z\overline{w}}{z-w}\right|,z,w\in\D$, we get $g_{\C\setminus K}(z,\infty)=\log\frac{\left|z-x_0\right|}{x_0-1}$, $z\in\D$. Thus,
	\[
	M_K=\sup\left\{\frac{\left|z-x_0\right|}{x_0-1}:\, z\in\D\right\}=\frac{x_0+1}{x_0-1}.
	\]
	
	\subsection{When $K$ is a line segment of the form $[1,x_0]$}
	In this case, we have the following:
	\begin{prop}\label{compu-alphaK-interval}Let $x_0>1$ and $K=[1,x_0]$. Then $\alpha_K=\frac{1}{\sqrt{x_0}-1}$.
	\end{prop}
	\begin{proof}We will make use of Corollary \ref{poisson}. Let $\psi:\D\to\overline{\C}\setminus[-2,2]$ be the function defined by
		$$\psi(z)=z+\frac{1}{z}\text{, for }z\in\D.$$
		It is well-known that $\psi$ is a conformal mapping onto $\overline{\C}\setminus[-2,2]$. Thus, by applying a linear transformation, the function $\tilde{\phi}:\D\to\overline{\C}\setminus[1,x_0]$ given by
		$$\tilde{\phi}(z)=\frac{x_0-1}{4}\psi(z)+\frac{x_0+1}{2}\text{, for }z\in\D,$$
		is a conformal mapping onto $\overline{\C}\setminus[1,x_0]$.
		
		For any $0<\varepsilon<1$, the ``interior'' of the Jordan curve $\tilde{\phi}\left(\partial D\left(0,1-\varepsilon\right)\right)$ bounds  a domain that contains $K$ (here and below, by ``interior'' we mean the bounded domain that has the curve as boundary, in accordance with the Jordan Curve Theorem). We denote by $G_\varepsilon$ this domain, and we may and shall assume that $\varepsilon$ is small enough, so that $G_\varepsilon\subseteq\C\setminus(-\infty,0]$. 
		
		Then $\tilde{\phi}$ maps $D(0,1-\varepsilon)$ conformally onto $\overline{\C}\setminus \overline{G_{\varepsilon}}$, and the function $\phi:\overline{\C}\setminus \overline{G_{\varepsilon}}\to\D$ defined by $\phi(z)=\frac{\tilde{\phi}^{-1}(z)}{1-\varepsilon}$ for $z\in\overline{\C}\setminus\overline{G_{\varepsilon}}$, is conformal onto $\D$. We note that $\tilde{\phi}(0)=\infty$, hence $\phi(\infty)=0$. Moreover, $\phi(0)=\frac{\tilde{\phi}^{-1}(0)}{1-\varepsilon}$. Now,
		\begin{align*}
			\begin{autobreak}
				\MoveEqLeft[0]
				(1+\alpha)P(\phi(\infty),\phi(\zeta))-\alpha P(\phi(0),\phi(\zeta))
				=\frac{1}{2\pi}\left((1+\alpha)-\alpha\frac{1-\left(\frac{\tilde{\phi}^{-1}(0)}{1-\varepsilon}\right)^2}{\left|\phi(\zeta)-\left(\frac{\tilde{\phi}^{-1}(0)}{1-\varepsilon}\right)\right|^2}\right)
				\geq\frac{1}{2\pi}\left((1+\alpha)-\alpha\frac{1-\left(\frac{\tilde{\phi}^{-1}(0)}{1-\varepsilon}\right)^2}{\left(1+\left(\frac{\tilde{\phi}^{-1}(0)}{1-\varepsilon}\right)\right)^2}\right)
				=\frac{1}{2\pi}\left((1+\alpha)-\alpha\frac{1-\frac{\tilde{\phi}^{-1}(0)}{1-\varepsilon}}{1+\frac{\tilde{\phi}^{-1}(0)}{1-\varepsilon}}\right)
			\end{autobreak}
		\end{align*}
		where the inequality comes from the fact that
		\[
		\frac{\tilde{\phi}^{-1}(0)}{1-\varepsilon}=\frac{\sqrt{\left(\frac{x_0+1}{x_0-1}\right)^2-1}-\frac{x_0+1}{x_0-1}}{1-\varepsilon}\in(-1,0),
		\]
		and, thus, that
		\[
		\min\left\{\left|\phi(\zeta)-\left(\frac{\tilde{\phi}^{-1}(0)}{1-\varepsilon}\right)\right|:\zeta\in\partial G_\varepsilon\right\}=1+\frac{\tilde{\phi}^{-1}(0)}{1-\varepsilon}.
		\]
		A straightforward computation shows that
		\[
		(1+\alpha)-\alpha\frac{1-\frac{\tilde{\phi}^{-1}(0)}{1-\varepsilon}}{1+\frac{\tilde{\phi}^{-1}(0)}{1-\varepsilon}}\geq 0
		\]
		if and only if $\tilde{\phi}^{-1}(0)\geq-\frac{1-\varepsilon}{1+2\alpha}$.
		Without loss of generality, we assume that $\varepsilon$ is small enough so that $\frac{1-\varepsilon}{1+2\alpha}<1$. Composing the previous inequality by $\tilde{\phi}$, that is decreasing on $(-1,0)$ (since $\psi$ is), one easily checks that $\tilde{\phi}^{-1}(0)\geq-\frac{1-\varepsilon}{1+2\alpha}$ holds if and only if
		\[
		x_0\leq \left(\frac{(1+2\alpha)+(1-\varepsilon)}{(1+2\alpha)-(1-\varepsilon)}\right)^2.
		\]
		Since the right-hand side of the last inequality is a continuous decreasing function of $\varepsilon$ on $(0,1)$, that takes the value $\left(\frac{1+\alpha}{\alpha}\right)^2$ at $0$, we deduce frome the above that, if $x_0<\left(\frac{1+\alpha}{\alpha}\right)^2$, or equivalently if $\alpha<\frac{1}{\sqrt{x_0}-1}$, then $(G_{\varepsilon},z^{\alpha})$ is a pair of approximation for every $\varepsilon$ small enough, by Corollary \ref{poisson}. Conversely, if $\alpha\geq \frac{1}{\sqrt{x_0}-1}$, the same argument shows that $(G_\varepsilon,z^{\alpha})$ is a pair of approximation for no $\varepsilon>0$. To finish, if $G$ is a simply connected bounded domain in $\C\setminus(-\infty,0]$ that contains $K$, then $G\supseteq G_{\varepsilon}$ for some $\varepsilon>0$, hence $(G,z^{\alpha})$ cannot be a pair of approximation. This completes the proof.
	\end{proof}
	Using Theorem \ref{Solynin} (taking into account its last part), we immediately get
	\begin{equation*}
		M_K
		=\exp\left(\Phi\left(\frac{\textup{dist}(-1,K)}{\textup{diam}(K)}\right)\right)
		=\exp\left(\Phi\left(\frac{2}{x_0-1}\right)\right)
		\left(\sqrt{\frac{2}{x_0-1}+1}+\sqrt{\frac{2}{x_0-1}}\right)^2.
	\end{equation*}
	
	\subsection{When $K$ is a subarc of the unit circle}
	Let $\theta_0\in(0,2\pi)$ and $K=\left\{e^{i\theta}:\theta\in\left[\frac{-\theta_0}{2},\frac{\theta_0}{2}\right]\right\}$. Let $\psi:\overline{\C}\setminus K\to\overline{\C}\setminus\overline{D\left(0,\sin\left(\frac{\theta_0}{4}\right)\right)}$ be the function defined by
	\[
	\psi(z)=\frac{1}{2}\left(z-1+\sqrt{\left(z-e^{i\frac{\theta_0}{2}}\right)\left(z-e^{-i\frac{\theta_0}{2}}\right)}\right),
	\]
	where the square root is chosen so that $\psi(z)=z+O(1)$ as $z\to\infty$. Then $\psi$ is a conformal mapping onto $\overline{\C}\setminus\overline{D\left(0,\sin\left(\frac{\theta_0}{4}\right)\right)}$ (see Exercise 4 on page 137 in \cite{Ransford1995}).
	
	Now, let $0<\varepsilon<1-\sin\left(\frac{\theta_0}{4}\right)$. Then $\psi^{-1}\left(\partial D\left(0,\sin\left(\frac{\theta_0}{4}\right)+\varepsilon\right)\right)$ is a Jordan curve whose ``interior'' contains $K$. Denoting by $G_\varepsilon$ this domain, it is clear that $G_\varepsilon\subseteq\C\setminus(-\infty,0]$. The following proposition provides us with the exact value of $\alpha_K$ in this setting.
	\begin{prop}
		For $K$ as above, the following holds:
		\[
		\alpha_K=\frac{1-\sin\left(\frac{\theta_0}{4}\right)}{2\sin\left(\frac{\theta_0}{4}\right)}.
		\]
	\end{prop}
	\begin{proof}
		Since $\psi$ maps $\overline{\C}\setminus \overline{G_{\varepsilon}}$ conformally onto $\overline{\C}\setminus\overline{D\left(0,\sin\left(\frac{\theta_0}{4}\right)+\varepsilon\right)}$, the function
		\[
		\phi:\left\{\begin{array}{ccc}
			\overline{\C}\setminus \overline{G_{\varepsilon}} & \to & \D\\
			z & \mapsto & \frac{\sin\left(\frac{\theta_0}{4}\right)+\varepsilon}{\psi(z)}
		\end{array}
		\right.
		\]
		is conformal onto $\D$. Taking into account that $\psi(\infty)=\infty$ and $\psi(0)=-1$ we have, for $\zeta\in\partial G_{\varepsilon}$,
		\begin{align*}
			\begin{autobreak}
				\MoveEqLeft[0]
				(1+\alpha)P(\phi(\infty),\phi(\zeta))-\alpha P(\phi(0),\phi(\zeta))
				=\frac{1}{2\pi}\left((1+\alpha)-\alpha\frac{1-\left(\sin\left(\frac{\theta_0}{4}\right)+\varepsilon\right)^2}{\left|\phi(\zeta)-\left(\sin\left(\frac{\theta_0}{4}\right)+\varepsilon\right)\right|^2}\right)
				\geq\frac{1}{2\pi}\left((1+\alpha)-\alpha\frac{1-\left(\sin\left(\frac{\theta_0}{4}\right)+\varepsilon\right)^2}{\left(1-\left(\sin\left(\frac{\theta_0}{4}\right)+\varepsilon\right)\right)^2}\right)
				=\frac{1}{2\pi}\left((1+\alpha)-\alpha\frac{1+\sin\left(\frac{\theta_0}{4}\right)+\varepsilon}{1-\left(\sin\left(\frac{\theta_0}{4}\right)+\varepsilon\right)}\right)
			\end{autobreak}
		\end{align*}
		where we used the fact that $\min\left\{\left|\phi(\zeta)-\left(\sin\left(\frac{\theta_0}{4}\right)+\varepsilon\right)\right|:\zeta\in\partial G_\varepsilon\right\}=1-\left(\sin\left(\frac{\theta_0}{4}\right)+\varepsilon\right)$ for the inequality above. Now, a direct computation shows that the inequality
		\[
		(1+\alpha)-\alpha\frac{1+\sin\left(\frac{\theta_0}{4}\right)+\varepsilon}{1-\sin\left(\frac{\theta_0}{4}\right)+\varepsilon}\geq0
		\]
		is equivalent to
		\[
		\sin\left(\frac{\theta_0}{4}\right)+\varepsilon\leq\frac{1}{1+2\alpha}.
		\]
		Thus, by Corollary \ref{poisson}, if $\sin\left(\frac{\theta_0}{4}\right)<\frac{1}{1+2\alpha}$, or equivalently if $\alpha<\frac{1-\sin\left(\frac{\theta_0}{4}\right)}{2\sin\left(\frac{\theta_0}{4}\right)}$, then $(G_\varepsilon,z^{\alpha})$ is a pair of approximation for every $\varepsilon$ small enough. Conversely, if $\alpha\geq \frac{1-\sin\left(\frac{\theta_0}{4}\right)}{2\sin\left(\frac{\theta_0}{4}\right)}$, then $(G_\varepsilon,z^{\alpha})$ is a pair of approximation for no $\varepsilon>0$. We conclude as at the end of the proof of Proposition \ref{compu-alphaK-interval}.
	\end{proof}
	Let us now estimate $M_K$. By Theorem \ref{Solynin},
	\[
	M_K\leq\sup\left\{\Phi\left(\frac{\textup{dist}(z,K)}{\textup{diam}(K)}\right):\, z\in\D\right\}=\Phi\left(\frac{\sup\left\{\textup{dist}(z,K):\, z\in\D\right\}}{\textup{diam}(K)}\right).
	\]
	Now, we easily see that
	\[ 
	\textup{diam}(K)= \left\{
	\begin{array}{ll}
		2\sin\left(\frac{\theta_0}{2}\right) & \text{ if } \theta_0\in(0,\pi] \\
		2 & \text{ else },
	\end{array} 
	\right. 
	\]
	and that $\sup\left\{\textup{dist}(z,K):\, z\in\D\right\}=\textup{dist}(-1,K)=\sqrt{1+2\cos\left(\frac{\theta_0}{2}\right)}$. Finally, we get
	\[ 
	M_K\leq \left\{
	\begin{array}{ll}
		\exp\left(\Phi\left(\frac{\sqrt{1+2\cos\left(\frac{\theta_0}{2}\right)}}{2\sin\left(\frac{\theta_0}{2}\right)}\right)\right) & \text{ if } \theta_0\in(0,\pi] \\
		\exp\left(\Phi\left(\frac{\sqrt{1+2\cos\left(\frac{\theta_0}{2}\right)}}{2}\right)\right) & \text{ else }.
	\end{array} 
	\right. 
	\]

	\bibliographystyle{amsplain}
	\bibliography{refs}
	
\end{document}